\documentclass[11pt]{article}
\usepackage{tocbibind} 
\usepackage{amsmath,amsthm,amssymb}
\usepackage{fancyhdr}
\usepackage[british]{babel}
\usepackage{geometry,mathtools}
\usepackage{enumitem}
\usepackage{algpseudocode}
\usepackage{dsfont}
\usepackage{centernot}
\usepackage{xstring}
\usepackage{colortbl}

\usepackage[section]{algorithm}
\usepackage{graphicx}
\usepackage[font={footnotesize}]{caption}
\usepackage[usenames,dvipsnames,table]{xcolor}
\usepackage{pstricks,tikz}
\usetikzlibrary{patterns,calc,fadings,decorations.pathreplacing}
\usepackage[h]{esvect}
\usepackage[
bookmarksopen=true,
bookmarksopenlevel=1,
colorlinks=true,
linkcolor=darkblue,
linktoc=page,
citecolor=darkblue,
urlcolor=darkblue
]{hyperref}
\usepackage{bbm}
\usepackage{cleveref}
\usepackage[sort]{cite}

\numberwithin{equation}{section}

\definecolor{darkblue}{rgb}{0,0,0.5}

\newdimen\margin
\def\textno#1&#2\par{
   \margin=\hsize
   \advance\margin by -4\parindent
          \setbox1=\hbox{\sl#1}
   \ifdim\wd1 < \margin
      $$\box1\eqno#2$$
   \else
      \bigbreak
      \hbox to \hsize{\indent$\vcenter{\advance\hsize by -3\parindent
      \it\noindent#1}\hfil#2$}
      \bigbreak
   \fi}

\newtheorem{theorem}[algorithm]{Theorem}

\newtheorem{lemma}[algorithm]{Lemma}

\theoremstyle{definition}

\newtheorem{defin}[algorithm]{Definition}
\newtheorem{remark}[algorithm]{Remark}

\newcounter{stepenv}
\newenvironment{stepenv}[1][]{\refstepcounter{stepenv}}{}

\newcounter{step}[stepenv]

\newcounter{substep}[step]
\renewcommand{\thesubstep}{\thestep.\arabic{substep}}

\newcounter{claim}

\def\In{\subseteq}

\def\COMMENT#1{}
\def\TASK#1{}
\let\TASK=\footnote             
\let\COMMENT=\footnote          

\author{
Nemanja Dragani\'c\thanks{
Mathematical Institute, University of Oxford, UK. \\\emph{Email}: \textbf{nemanja.draganic@maths.ox.ac.uk}.
Research supported by SNSF project 217926.
}
\and
Michael Krivelevich\thanks{
†School of Mathematical Sciences, Tel Aviv University, Tel Aviv 6997801, Israel. \\
\emph{Email}: \textbf{krivelev@tauex.tau.ac.il}.
Research supported in part by NSF-BSF grant 2023688. 
}}
\title{Disjoint connected dominating sets in pseudorandom graphs}
\begin{document} 
\maketitle
\begin{abstract}
A connected dominating set (CDS) in a graph is a dominating set of vertices that induces a connected subgraph. Having many disjoint CDSs in a graph can be considered as a measure of its connectivity, and has various graph-theoretic and algorithmic implications.
We show that $d$-regular (weakly) pseudoreandom graphs contain $(1+o(1))d/\ln d$ disjoint CDSs, which is asymptotically best possible. In particular, this implies that random $d$-regular graphs typically contain $(1+o(1))d/\ln d$ disjoint CDSs. 
\end{abstract}

\section{Introduction}

Connected dominating sets (CDSs) are an important concept from graph theory with many practical applications.
A CDS in a graph is a subset of vertices that forms a dominating set (every vertex is either in the set or adjacent to a vertex in the set) and induces a connected subgraph. 
These sets are important for tasks such as efficient routing~\cite{cheng2006virtual,wu2003broadcasting}, backbone formation~\cite{ephremides1987design,chang2004routing}, and load balancing in networks~\cite{chen2001span,deb2003multi}, providing a compact and connected subset that helps maintain overall network functionality (see the survey~\cite{yu2013connected}).
Computing the size of the smallest CDS of a graph is a well-known NP-hard problem (see~\cite{yu2013connected}) and it is also hard to approximate.

Another important avenue of research involves finding small CDSs in graphs with specific constraints. Finding a small CDS in a graph is equivalent to finding a spanning tree with few non-leaves, and the results in the current paragraph have historically used the latter language. Generalizing and improving results of
Linial and Sturtevant (see \cite{bonsma2012improved}), Kleitman and West \cite{kleitman1991spanning}, Griggs and Wu \cite{griggs1992spanning} and of
Caro, West and Yuster~\cite{caro2000connected},
Alon~\cite{alon2023spanning} showed that every connected $n$-vertex graph with minimum degree at least $d$ contains a CDS of size at most $\frac{ n}{
d + 1} (\ln(d + 1) + 4) - 2,$ or equivalently, a spanning tree with that few non-leaves. This is asymptotically best possible, as random $d$-regular graphs do not contain a CDS of size $(1-o_d(1))\frac{ n}{d} \ln d  $; for more history on this problem see~\cite{alon2023spanning}. 

In many applications, it is important to have multiple (vertex-disjoint) CDSs in a graph. Partitions and packings of CDSs play a key role in the design of wireless networks, where a CDS acts as a virtual backbone, enabling other nodes to exchange messages and route traffic efficiently~\cite{das1997routing,das1997routing2}. 
Additionally, work by Censor-Hillel, Ghaffari and Kuhn~\cite{censor2014new} highlights a strong connection between optimal packings of CDSs and the throughput of store-and-forward routing algorithms in wireless networks. As we will discuss below, multiple CDSs are important in the study of the connectivity of graphs~\cite{censor2014new,censor2017tight};
in fact, if a graph contains $k$ disjoint CDSs, then it is $k$-connected. 
Multiple disjoint CDSs have also recently been exploited in the context of graph rigidity~\cite{krivelevich2023rigid}.

In the context of the aforementioned result of Alon, if one aims to find several vertex disjoint CDSs in a connected graph, a minimum degree assumption is not enough. Indeed, two cliques of size $k$ sharing exactly one vertex cannot have two disjoint connected dominating sets, as each such set must contain the shared vertex. Evidently, the problem with this example is low connectivity, so it is natural to ask what happens for $k$-connected graphs. 

As we noticed, $k$ disjoint CDSs imply $k$-connectivity.
The converse direction is not only interesting in its own right, but it is related to the study of vertex-connectivity of random subgraphs of $k$-connected graphs, as demonstrated by Censor-Hillel, Ghaffari, and Kuhn~\cite{censor2014new} and in the subsequent improvement by Censor-Hillel, Ghaffari, Giakkoupis, Haeupler, and Kuhn~\cite{censor2017tight}. In the latter, it was shown that $k$-connected $n$-vertex graphs contain at least $\Omega(k/\log^2n)$ disjoint CDSs. As explained in~\cite{censor2014new}, this result can be viewed as an analog of the classic results of Tutte and Nash-Williams postulating that $k$-edge connected graphs contain at least $(k-1)/2$ edge-disjoint spanning trees. Instead of “decomposing" the edge-connectivity into spanning trees, the vertex-connectivity is decomposed into disjoint CDSs.
Perhaps surprisingly, in \cite{censor2014new} it was shown that there exist $k$-connected graphs with at most $O(k/\log n+1)$ disjoint CDSs, for all $k\in[n/4]$. In particular, they also show that the connectivity has to grow at least logarithmically with $n$ to guarantee at least $2$ disjoint CDSs.

An important class of highly connected graphs are pseudorandom graphs, and in particular $(n, d, \lambda)$-graphs, introduced by Alon (see the survey~\cite{krivelevich2006pseudo}). 
Those are $n$-vertex $d$-regular graphs, characterized by their spectral properties.

\begin{defin}
An $n$-vertex graph $G$ whose adjacency matrix has eigenvalues $\lambda_1\geq \lambda_2\geq \ldots\geq \lambda_n$ is an $(n,d,\lambda)$-\emph{graph} if it is $d$-regular and satisfies: $\max\{\lambda_2,|\lambda_n|\}\leq \lambda$.
\end{defin}
We refer readers who are less familiar with this class of graphs to~\Cref{sec:ndlambda}, and in particular to~\Cref{lem:EML} (the so called Expander Mixing Lemma) giving a tangible property of $(n,d,\lambda)$-graphs which will be sufficient for the purposes of our study.
These graphs exhibit strong connectivity and expansion properties, making them ideal candidates for applications in network design requiring high robustness against failures \cite{hoory2006expander}. Even though those graphs are characterized by a deterministic property, they resemble random graphs in many important aspects.
As we will exhibit in~\Cref{sec:ndlambda}, the \emph{spectral ratio} $d/\lambda$ controls how close the graph is to a random graph: the larger it is, the closer the edge distribution is to that of a random graph. 

High spectral ratio indicates that the graph is a robustly expanding graph, which intuitively should support many disjoint CDSs, as such graphs remain well-connected even when certain large subsets of vertices are removed. As we have seen, just high connectivity does not necessarily guarantee even two disjoint CDSs, so it is natural to try and understand whether a type of \emph{robust connectivity} as in $(n, d, \lambda)$-graphs implies the existence of many CDSs.

In a random $n$-vertex $d$-regular graph, every dominating set has at least $(1+o_d(1))\frac{n\ln d}{d}$ vertices \cite{alon1990transversal}, with high probability. 
Hence, they typically can contain at most $(1+o_d(1))\frac{d}{\ln d}$ disjoint CDSs.
Almost certainly, such graphs are $(n, d, \lambda)$-graphs with essentially the largest possible spectral ratio ($\lambda\sim \sqrt{d}$, see~\cite{sarid2023spectral}), hence $(1+o_d(1))\frac{d}{\ln d}$ disjoint CDSs is a natural upper bound to attain in $(n,d,\lambda)$-graphs, possibly even with more modest assumptions on the spectral ratio. And indeed, in this work we show that that very modest asumptions on the spectral ratio of $(n,d,\lambda)$-graphs already guarantee the (asymptotically) optimal number of disjoint CDSs.

\begin{theorem}\label{thm:main}
For every $\varepsilon>0$, there exists a $C=C(\varepsilon)$ so that the following holds.
Every $(n,d,\lambda)$-graph with $d/\lambda>C$ contains at least $(1-\varepsilon)\frac{d}{\ln d}$ disjoint connected dominating sets.
\end{theorem}

In particular, this implies that random $d$-regular graphs contain $(1+o_d(1))\frac{d}{\ln d}$ disjoint CDSs. Furthermore, essentially the same proof, along with standard concentration bounds for vertex degrees, can be used to show that the binomial random graph $G(n, \frac{d}{n})$ contains $(1 + o_d(1)) \frac{d}{\ln d}$ CDSs for $d = \omega(\ln n)$. This result is optimal both in terms of the number of CDSs and the asymptotics of the edge probability. Specifically, when $d\ll \log n$, typically there are isolated vertices in $G(n, \frac{d}{n})$, and thus there are no CDS in the graph. 

\Cref{thm:main} also has implications related to the well-known Independent Spanning Trees conjecture by Itai and Zehavi~\cite{zehavi1989three} (see, e.g., the survey~\cite{cheng2023independent} for historical background). The conjecture asserts that every \(k\)-connected graph $G$ contains \(k\) vertex-independent trees; these are \(k\) spanning trees rooted at a vertex \(r \in V(G)\), such that for each vertex \(v \in V(G)\), the paths from \(r\) to \(v\) in different trees are internally vertex-disjoint. So far, the best known
approximation of this 1989 conjecture is given in \cite{censor2017tight}, and gives $\Omega(k/\log^2n)$ vertex-independent trees. It is easy to verify that a graph $G$ possessing $k$ disjoint CDSs contains $k$ vertex-independent trees. Hence, our result implies that $(n,d,\lambda)$-graphs with mild assumptions on the spectral ratio contain $(1 + o_d(1)) \frac{d}{\ln d}$ vertex-independent trees. We include some brief comments on this conjecture in the concluding remarks.

The decision problem whether a graph contains at least two disjoint CDSs is NP-complete. 
As we will see in our proof, in the case of $(n,d,\lambda)$-graphs, there is a polynomial time algorithm to find $(1 + o_d(1)) \frac{d}{\ln d}$ disjoint such sets. 


\section{Preliminaries}
\textbf{Notation.} Let $G$ be a graph, and $A,B$ be subsets of its vertices. We denote by $e(A,B)$ the number of ordered pairs $(a,b)\in A\times B$ of edges in $G$.
We denote by $\Gamma_B(A)$ the vertices in $B$ that have at least one neighbor in $A$, and by $N_B(A)$ the external neighborhood in $B$, i.e., the vertices in $\Gamma_B(A)\setminus A$. 
We also write $\Gamma_G(A)$ and $N_G(A)$ when $B=V(G)$.
We omit the subscripts when they are clear from the context.
Given a set $S$, we denote by $I[S]$ the graph with vertex set $S$, isomorphic to an independent set. A linear forest is a graph that is a vertex-disjoint union of paths.
We systematically omit rounding signs for clarity of presentation, where it does not impact the argument.

\subsection{Properties of $(n,d,\lambda)$-graphs}\label{sec:ndlambda}

The following is the well known Expander Mixing Lemma, which states that the edge distribution in $(n,d,\lambda)$-graphs is close to that of a random graph -- the smaller $\lambda$ is, the closer the edge distribution is to what we would expect in a random graph.
\begin{lemma}[Expander Mixing Lemma]\label{lem:EML}
Let $G$ be a $(n,d,\lambda)$-graph, and $A,B$ two vertex subsets. Then the following holds:
\[
\left|e(A,B)-\frac{|A||B|d}{n}\right|\leq \lambda\sqrt{|A||B|}
\]
\end{lemma}

The next lemma is a simple consequence of the Expander Mixing Lemma. It states that there is always an edge between two large enough vertex sets, and that a subgraph of an $(n,d,\lambda)$-graph of high minimum degree has strong expansion properties. 

\begin{lemma}\label{lem:expansion}
Let $G$ be an $(n,d,\lambda)$-graph. Then the following hold:
\begin{enumerate}[label=(\alph*)]
\item\label{p:joint} For every two disjoint sets of vertices $X,Y$ of size $|X|,|Y|>\frac{\lambda n}{d}$ it holds that $e(X,Y)>0$.
\item\label{p:expansion} Fix $0<\varepsilon\leq 1$. Let $B\In V(G)$ be such that every vertex has at least $\varepsilon d/3$ neighbors in $B$. Then for every $k$ with $1<k< (\frac{d}{12\lambda})^2$ and for every subset $X\subseteq V(G)$ of size at most $\frac{n}{12k}$ it holds that $|\Gamma_B(X)|\geq \varepsilon^2k|X|$.
\end{enumerate}
\end{lemma}

\begin{proof}
\begin{enumerate}[label=(\alph*)]

\item By the Expander Mixing Lemma, we get that $$e(A,B)\geq \frac{|X||Y|d}{n}-\lambda\sqrt{|X||Y|}>\frac{\lambda^2n}{d}-\frac{\lambda^2n}{d}=0, $$
which finishes the proof.

\item Let $X$ be of size $|X|\leq n/12k$, and suppose for contradiction that $|\Gamma_B(X)|< \varepsilon^2 k|X|\leq \varepsilon^2 n/12$.
The number of edges between $X$ and $B$ is at least $\varepsilon|X|d/6$. On the other hand, by the Expander mixing lemma, we have that 
$$e(X,\Gamma_B(X))\leq |X||\Gamma_B(X)|\frac{d}{n}+\lambda \sqrt{|X||\Gamma_B(X)|}\leq \frac{\varepsilon^2|X|d}{12}+\varepsilon\lambda |X|\sqrt{k}< 2\cdot \frac{\varepsilon|X|d}{12}\leq \frac{\varepsilon|X|d}{6},$$
a contradiction.
\end{enumerate}
\end{proof}

\begin{remark}
It is known and trivial that we always have $\lambda=\Omega(\sqrt{d})$ when, say, $d<0.99n$~\cite{krivelevich2006pseudo}. 
The $(n,d,\lambda)$-graphs with $\lambda=\Theta(\sqrt{d})$ have best possible expansion properties, and as can be seen from~\Cref{lem:expansion}, small vertex sets expand by a factor of $\Theta(d)$ in those graphs. Though random $d$-regular graphs are well-known examples of such graphs, there exist several explicit constructions (see \cite{krivelevich2006pseudo}); one example is the famous Ramanujan graphs constructed by Lubotzky, Phillips and Sarnak \cite{lubotzky1988ramanujan}; see also~\cite{alon2021explicit} for constructions for all possible choices of $d$ and $n$.
\end{remark}

\subsection{The Friedman-Pippenger tree embedding technique with rollbacks}

In this section, we introduce our main tool: the Friedman-Pippenger tree embedding technique, also known as the extendability method. 
Building on the tree embedding results of Friedman and Pippenger~\cite{friedman1987expanding} and of Haxell~\cite{haxell2001tree}, this technique enables the robust construction of linear-sized trees in expander graphs. The key result we rely on is~\Cref{thm:FP}, which facilitates the construction of such trees. Additionally, a notable enhancement to this method is provided by the simple~\Cref{lemma:delete}, which allows us to remove leaves from the currently constructed tree while preserving the extendability property. This preservation is crucial for the repeated application of~\Cref{thm:FP}. 
This method has been pivotal in settling many long-standing open problems in graph theory; see, for example, \cite{draganic2022rolling,montgomery2019spanning, draganic2024hamiltonicity}.
We will use the machinery presented in \cite{montgomery2019spanning} -- further details are provided below. Before we state the results, we give two useful definitions. The first one describes the notion of expansion we will use.
\begin{defin}\label{def:expanding}
	Let $s \in \mathbb{N}$ and $K > 0$. We say that a graph $G$ is \emph{$(s, K)$-expanding} if for every subset $X \subseteq V(G)$ of size $|X| \le s$ we have $|N_G(X)| \ge K |X|$.
\end{defin}

We also need the following notion 
\begin{defin}
A graph $G$ is $m$-\emph{joined} for $m>0$, if for every two disjoint vertex subsets $A,B\subseteq V(G)$ of size at least $m$ we have that $e(A,B)>0$.
\end{defin}

We also need the notion of an $(m,D)$-extendable embedding from \cite{montgomery2019spanning}.
\begin{defin}\label{deF:goodness}
Let $m, D \in \mathbb{N}$ satisfy $D \geq 3$ and $m \geq 1$, let $G$ be a graph, and let $S \subset G$ be a subgraph of $G$. We say that $S$ is $(m, D)$-\emph{extendable} if $S$ has maximum degree at most $D$ and
\begin{equation}\label{eq:extendable}
|\Gamma_G(U) \setminus V(S)| \geq (D - 1)|U| - \sum_{x \in U \cap V(S)} (d_S(x) - 1)
\end{equation}
for all sets $U \subset V(G)$ with $|U| \leq 2m$.
\end{defin}

The next theorem is the key technical tool in the section. It allows us to extend an $(m,D)$-extendable embedding of a graph by attaching to one of its vertices a tree of certain size and maximum degree at most $D/2$.
\begin{theorem}[Corollary 3.7  in \cite{montgomery2019spanning}] \label{thm:FP}
Let $m, D \in \mathbb{N}$ satisfy $D \geq 3$ and $m \geq 1$, and let $T$ be a tree with maximum degree at most $D/2$, which contains the vertex $t \in V(T)$. Let $G$ be an $m$-joined graph and suppose $R$ is an $(m, D)$-extendable subgraph of $G$ with maximum degree $D/2$. Let $v \in V(R)$ and suppose $|R| + |T| \leq |G| - 2Dm - 3m$. Then, there is a copy $S$ of $T$ in $G - (V(R) \setminus \{v\})$, in which $t$ is copied to $v$, such that $R \cup S$ is $(m, D)$-extendable in $G$.
\end{theorem}

The following result shows that adding an edge between two vertices of an $(m,D)$-extendable graph maintains its extendability.
\begin{lemma}[Lemma 3.9 in \cite{montgomery2019spanning}] \label{lemma:non-leaf edge}
Let $m, D \in \mathbb{N}$ satisfy $D \geq 3$ and $m \geq 1$, let $G$ be a graph, and let $S$ be an $(m, D)$-extendable subgraph of $G$. If $s, t \in V(S)$ with $d_S(s), d_S(t) \leq D - 1$ and $st \in E(G)$, then $S + st$ is $(m, D)$-extendable in $G$.
\end{lemma}

Finally, the last result we need is another simple corollary of the definition of $(m,D)$-extendability which allows for leaf removal.

\begin{lemma}[Lemma 3.8 in \cite{montgomery2019spanning}] \label{lemma:delete}
Let $m, D \in \mathbb{N}$ satisfy $D \geq 3$ and $m \geq 1$, let $G$ be a graph, and let $S$ be a subgraph of $G$. Furthermore, suppose there exist vertices $s \in V(S)$ and $y \in N_G(S)$ so that the graph $S + ys$ is $(m, D)$-extendable. Then $S$ is $(m, D)$-extendable. \hfill $\square$
\end{lemma}

\begin{remark}
Let us remark here that \Cref{thm:FP}, \Cref{lemma:non-leaf edge} and \Cref{lemma:delete} are not algorithmic in the form they are stated. But, in \cite{draganic2022rolling} an algorithmic version is presented for $(n,d,\lambda)$-graphs. Hence, every step of the proof which uses these two results can also be made into a polynomial time algorithm. We will also use the Lovasz Local Lemma in a few instances to prove the existence of certain sets -- these too can be made into a deterministic polynomial time algorithm~\cite{moser2010constructive}.
\end{remark}

\section{Proof of~\Cref{thm:main}}

For technical reasons we assume that $\varepsilon<1/1000$, and we let $C=C(\varepsilon)$ be a large enough constant, such that all inequalities below go through. Note that since $d\geq \lambda C$, we also may assume that $d$ is large enough.
We aim to find $$d^*=\frac{(1-\varepsilon)d}{\ln d}$$ disjoint connected dominating sets.
We start with a lemma that provides us with $d^*$ disjoint dominating sets that are not necessarily connected, but do not have too many connected components, and such that every vertex in their union has large neighborhood in a disjoint set of vertices of linear size. Let us note here that we do not use the assumption on $\lambda$ in the lemma, so its conclusion also holds for arbitrary $d$-regular graphs $G$ on $n$ vertices.

\begin{lemma}\label{lem:small number of components}
There exist disjoint subsets $B, S_1,S_2,\ldots, S_{d^*}\subseteq V(G)$ such that:
\begin{itemize}
\item each $S_i$ is a dominating set in $G$;
\item $G[S_i]$ has at most $20n/(\varepsilon^2d)$ connected components;
\item every vertex in $V(G)$ has at least $\varepsilon^2 d/2$ neighbors in $B$.
\end{itemize}
\end{lemma}

\begin{proof}
If instead of $d^*$ one wanted to find $d^*/1000$ sets with the properties as above, then a simple random partition of the vertices with a standard application of the Local Lemma would be sufficient to prove the statement. Unfortunately, this does not work in our setting, as the number of dependencies outweighs the probability of a random set $S_i$ to be dominating. 

To remedy the situation, we will perform a random coloring of the vertices in two stages, with consecutive applications of the Local Lemma. To this end, denote 

\begin{alignat*}{2}
r_1&=\frac{d^*}{\ln ^5 d}=\frac{(1-\varepsilon)d}{\ln^6d};\quad &&r_2=\ln^5 d; \\
p_1&=\frac{1-\varepsilon^2}{r_1}=\frac{(1+\varepsilon)\ln^6d}{d}; \qquad &&p_2=\frac{1}{r_2}=\frac{1}{\ln^5d}.
\end{alignat*}
Noting that $r_1r_2=d^*$, we consider each set $S_i$ as corresponding to precisely one of the colors in $[r_1]\times[r_2]$. 

As a first step, we color each vertex $v\in V(G)$ either into one of the colors $[r_1]$ with probability 
$p_1$, or place it in the set $B$ with probability $\varepsilon^2$, independently. 
Now we can apply the Local Lemma (see, e.g., Chapter~5 of~\cite{alon2016probabilistic}) --- define the event $A_{v,c}$, for every $v\in V(G)$ and $c\in [r_1]\cup \{B\}$, to be the outcome where the number $k_c(v)$ of $c$-colored vertices in $N(v)$ satisfies
$$k_c(v)\notin\big[(1-\varepsilon/2)\mathbb E[k_c(v)], (1+\varepsilon/2)\mathbb E[k_c(v)]\big].$$
For $c\in [r_1]$ we have that $\mathbb E[k_c(v)]=(1+\varepsilon)\ln^6d$, and for $c=B$ the expectation is $\mathbb E[k_c(v)]=\varepsilon^2d\geq \ln^6d$. Hence,
we have by Chernoff bounds that $P(A_{v,c})\leq e^{-\ln^5d}=o(d^4)$.

The event $A_{v,c}$ possibly depends only on the events $A_{w,\hat{c}}$ for which either $v=w$, or for which $v$ and $w$ have a common neighbor, so
each event depends on at most $(d^2+1)(r_1+1)\leq d^3$ other events. Thus, by the Local Lemma, there is a choice of coloring so that none of $A_{v,c}$ holds --- in particular, every vertex has at least $(1-\varepsilon/2)(1+\varepsilon)\ln^6d\geq (1+\varepsilon/3)\ln^6d$ vertices in color $c\in[r_1]$ in its neighborhood, as well as at least $\varepsilon^2d/2$ vertices that got assigned to $B$.

In the second round, we assign another color to each vertex that got a color $c\in[r_1]$. More precisely, assign to each such vertex a color $c'\in[r_2]$ with probability $p_2=1/r_2$ independently; let $(c,c')$ be its final color. For each $v\in V(G)$, and $(c,c')\in[r_1]\times[r_2]$, denote by $A_{v,c,c'}$ the event that $N(v)$ contains at most $\varepsilon^2\ln d$, or at least $20\ln d$ vertices that got color $(c,c')$. Denote by $N_c(v)$ the neighborhood of $v$ that had color $c$ in the first coloring.

Since the event $A_{v,c,c'}$ possibly depends only on events $A_{w,c_0,c_0'}$ for which we have $c=c_0$ and $w\in \Gamma_G(N_c(v))$, it is independent of all but at most $k_c(v)\cdot d=O(d\ln^6d)$ other events. The probability that $N(v)$ contains at least $20\ln d$ vertices of color $(c,c')$ is by Chernoff bounds at most $o(d^{-2})$.
Furthermore, the probability that it contains at most $s=\varepsilon^2\ln d$ such vertices is at most
\begin{align*}
&\sum_{i=0}^{s} \binom{k_c(v)}{i} p_2^i \left(1 - p_2\right)^{k_c(v) - i} \\
&\leq (s+1) \cdot \binom{k_c(v)}{s} p_2^{s} \left(1 - p_2\right)^{k_c(v) - s}\\
&\leq (s+1)  \left(\frac{k_c(v)\cdot e}{s}\right)^s \left(\frac{1}{\ln^5d}\right)^{s} \left(1 - \frac{1}{\ln^5d}\right)^{k_c(v) - s}\\ 
&\leq (s+1)  \left(\frac{2e\ln^6d}{\varepsilon^2\ln d}\right)^s \left(\frac{1}{\ln^5d}\right)^{s} \left(1 - \frac{1}{\ln^5d}\right)^{(1+\varepsilon/4)\ln^6d}\\ 
&\leq (\varepsilon^2\ln d+1) \left(\frac{2e}{\varepsilon^2}\right)^{\varepsilon^2\ln d}  e^{-(1+\varepsilon/4)\ln d}
\leq e^{20\varepsilon^2\ln(\frac{1}{\varepsilon})\ln d}  e^{-(1+\varepsilon/4)\ln d}=o(d^{-1-\varepsilon/8}).
\end{align*}
where for the first inequality we used that in the binomial distribution the probability of observing a specific number of successes increases as the number of successes gets closer to the expected value, and also that for integers $a\geq b>0$ we have $\binom{a}{b}\leq (ae/b)^b$, as well as $1-x\leq e^{-x}$ for all $x\in\mathbb R$.
To summarize, $P(A_{v,c,c'})=o(d^{-1-\varepsilon/8})$ and each event $A_{v,c,c'}$ depends on at most $O(d\ln^6d)$ other events. 
Thus we can apply the Local Lemma again, to get that there is a choice of coloring such that 
for every vertex $v\in V(G)$ and each color $(c,c')\in[r_1]\times[r_2]$, the event  
$A_{v,c,c'}$ does not hold.
This gives the required partition. Indeed, for a given $i\in [d^*]$, let $(c,c')$ be the color corresponding to $S_i$; then we have

\begin{itemize}
\item $S_i$ is dominating as every vertex has a vertex of color $(c,c')$ in its neighborhood. 
\item Denote by $d_{c,c'}(v)$ the number of $(c,c')$-colored vertices in $N(v)$. Then the number of vertices in $S_i$ can be estimated as $\frac{\sum_{v_\in V(G)} d_{(c,c')}(v)}{d}\leq \frac{20n\ln d}{d}$, as every vertex of color $(c,c')$ in the sum is counted $d$ times, exactly once for each of its $d$ neighbors. Since
$G[S_i]$ has minimum degree $\varepsilon^2\ln d$, it has
at most $|S_i|/\delta(G[S_i])\leq 20n/(\varepsilon^2d)$ connected components.
\item By construction, every vertex has at least $\varepsilon^2 d/2$ neighbors in $B$.
\end{itemize}
\end{proof}

Now, we need to connect the components inside each $S_i$, and we will do so by using the vertices in $B$. The main challenge here is to construct short internally disjoint paths connecting the mentioned components, and whose interior is in $B$.

To this end, pick a representative vertex $x$ in each component in each $G[S_i]$. Denote by $X_i$ the set of representative vertices in $S_i$. 
Our goal is, for each $i\leq d^*$, to construct a path which contains all $x\in X_i$, with the remaining vertices being in $B$. Denote $X=\cup_{i\leq d^*} X_i$ and recall that by~\Cref{lem:small number of components} we have that $|X_i|\leq 20n/(\varepsilon^2d)$. For technical reasons, we may also assume that $|X_i|\geq n/(2d)$, by possibly adding more vertices from $S_i$ to $X_i$; since $S_i$ is dominating and $G$ is $d$-regular, this is indeed possible as then $|S_i|\geq n/(d+1)\geq  n/(2d)$.

We consider the graph $G'=G[X\cup B]$. Denote
\[
m=\frac{\lambda n}{d}+1 \quad \text{and} \qquad D=\frac{\varepsilon^4d}{36\lambda}.
\]
By \Cref{lem:expansion}\ref{p:joint}, $G'$ is $m$-joined.
Furthermore, the subgraph $I[X]$ (the independent set on $X$) is an $(m,D)$-extendable subgraph of $G'$. Indeed, by~\Cref{lem:expansion}\ref{p:expansion} for $k=d/(36\lambda)$ we see that for every set $U$ of size at most $n/12k>2m$ we have $|\Gamma_{G'}(U)\setminus X|\geq |\Gamma_B(U)|\geq \varepsilon^4\frac{d}{36\lambda}|U|= D|U|$, so inequality (\ref{eq:extendable}) is trivially satisfied for all $U$ with $|U|\le 2m$.

We will now show how to construct the required paths for $S_1$, and later we will argue how to do the same for the remaining $S_i$'s analogously. Denote by $$s=n-2Dm-3m\geq n-\varepsilon^4n-3\lambda n/d-3\geq n/2$$ the lower bound on the maximal number of vertices in an $(m,D)$-extendable subgraph of $G'$ which we are allowed to construct using \Cref{thm:FP}.

 Now, we can extend the $(m,D)$-extendable subgraph $I[X]$ by attaching $(D/2{-}1)$-ary trees of size $\frac{s}{3|X_1|}\leq d$ (and of depth at most $\log_{D/2{-}1 } d$) to each vertex in $X_1$. Note that this can be done using~\Cref{thm:FP}, since the total size of the constructed forest (including all vertices in $X$) is at most $s/3+|X|\leq s/2$. Now, consider the set $L_1$ of the vertices of an arbitrary choice of $\lfloor|X_1|/2\rfloor$ of the trees, and $L_2$ the vertices of the remaining trees. Since both sets contain at least $s/9> \lambda n/d$ vertices each, there is an edge between the two sets by~\Cref{lem:expansion}\ref{p:joint}. Adding this edge to the current $(m,D)$-extendable graph creates a path $P$ of length at most $2\log_{D/2{-}1 } d+1$ between two vertices in $X_1$, while the constructed graph stays $(m,d)$-extendable by \Cref{lemma:non-leaf edge}. Now we can use Lemma~\ref{lemma:delete} to remove all vertices (leaf-by-leaf) of the attached trees, except those in the path $P$, so that we are again left with an $(m,D)$-extendable subgraph.

We now repeat this process, until we have connected all vertices in $X_1$. More precisely, we perform the following inductively described procedure. Suppose that at the $i$-th step, for $i<|X_1|-1$, we have constructed internally vertex-disjoint paths $P_1,P_2,\ldots, P_i$ with endpoints in $X_1$, that satisfy the following invariants:
\begin{itemize}
\item the internal vertices are in $B$; 
\item the length of $P_j$ for all $j\leq i$ is at most $2\log_{D/2{-}1 } \left(\frac{n}{|X_1|-j+1}\right)+1$;
\item The graph $I[X]\cup P_1\cup\ldots \cup P_i$ is a linear forest and an $(m,D)$-extendable subgraph of $G'$;
\item the number of components (i.e., paths) in $I[X_1]\cup P_1\cup\ldots \cup P_i$ is $|X_1|-i$.
\end{itemize}
In the $(i+1)$-th step we want to construct a path $P_{i+1}$ that maintains the invariants. 
In the end, after the $(|X_1|-1)$-th step, we will then have the required path which contains all vertices in $X_1$.

For each path in the linear forest $I[X_1]\cup P_1\cup\ldots \cup P_i$, choose one endpoint, and let $X_1'$ denote the set of these endpoints. As before, extend the $(m,D)$-extendable subgraph $I[X]\cup P_1\cup\ldots \cup P_i$ by attaching a $(D/2{-}1) $-ary tree to each vertex in $X_1'$, each one of size $\frac{s}{3|X_1'|}$, and hence of depth at most $\log_{D/2{-}1 }\left(\frac{s}{|X_1'|}\right)\leq\log_{D/2{-}1 }  \left(\frac{n}{|X_1|-i}\right)$. We can indeed use~\Cref{thm:FP}, as the current forest consists of $X$ and the now constructed trees (which together contain at most $s/2$ vertices), and the previously found paths $P_1,\ldots, P_i$, whose total number of vertices is at most: 

\begin{align}
\sum_{1\leq j\leq i}\left(2\log_{D/2{-}1 } \left(\frac{n}{|X_1|-j+1}\right)+1\right) 
&\leq |X_1|+2\sum_{j=1}^{|X_1|-2}\log_{D/2{-}1 } \left(\frac{n}{|X_1|-j+1}\right)\notag \\
&= |X_1|+2\sum_{j=3}^{|X_1|}\log_{D/2{-}1 } \left(\frac{n}{j}\right)\notag \\
&=|X_1|+2|X_1|\log_{D/2{-}1 } n-2\sum_{j=3}^{|X_1|}\log_{D/2{-}1 } (j)\notag \\
&=|X_1|+2|X_1|\log_{D/2{-}1 } n-2\log_{D/2{-}1 } \left((|X_1|)!/2\right)\notag \\
&\leq |X_1|+2|X_1|\log_{D/2{-}1 } n-2|X_1|\log_{D/2{-}1 } \left(\frac{|X_1|}{e}\right)\notag \\
&\leq |X_1|+2|X_1|\left(\log_{D/2{-}1 } n-\log_{D/2{-}1 } \left(\frac{n}{2de}\right)\right)\notag\\
&=|X_1|+2|X_1|\log_{D/2{-}1 } \left(2de\right)\leq \frac{100n\ln d} {\varepsilon^2d\ln C}\leq \frac{s}{2},\label{total number in paths}
\end{align}
where we used that $n/(2d)\leq |X_1|\leq 20n/(\varepsilon^2d)$, $D/2{-}1\geq \sqrt{C}$ and $s\geq n/2.$
Hence, in total, the forest at any point contains less than $s$ vertices, so we can successfully apply~\Cref{thm:FP}.

As $i<|X_1|-1$, by the induction hypothesis we have at least $|X_1|-i\geq 2$ trees, so we can group them into two collections each containing at least a third of all vertices. As there are $s/3$ vertices in total in all trees, there is an edge between the vertices in these two collections, as the number of them is at least $\frac{s/3}{3}\geq \frac{\lambda n}{d}$ each. This edge closes a path $P_{i+1}$ connecting two endpoints in the linear forest $I[X_1]\cup P_1\cup\ldots \cup P_i$. Adding this edge via \Cref{lemma:non-leaf edge} and removing the unused edges (leaf-by-leaf) using~\Cref{lemma:delete}, we are left with an $(m,D)$-extendable subgraph 
$I[X_1]\cup P_1\cup\ldots \cup P_{i+1}$ of $G'$.
Let us verify that all invariants are still satisfied. 

By construction, all internal vertices of the new path are in $B$.
The length of the new path is at most twice the depth of the trees plus one, i.e. it is of length at most $2\log_{D/2{-}1}  \left(\frac{n}{|X_1|-i}\right)+1$.
By construction, the subgraph $I[X_1]\cup P_1\cup\ldots \cup P_i\cup P_{i+1}$ is still $(m,D)$-extendable. Furthermore, as $P_{i+1}$ connects two endpoints of a linear forest through a path which is internally vertex-disjoint from $I[X_1]\cup P_1\cup\ldots \cup P_i$, we obtain a new linear forest with one less component than before, resulting in $|X_1|-i-1$ components/paths, as required. 

Continuing from the final $(m,D)$-extendable embedding (which now connects all vertices in $X_1)
$, we can now follow the same procedure for the sets $X_2,\ldots, X_{d^*}$, one by one, always extending the current $(m,D)$-extendable subgraph. Here we stress that when we are about to do the procedure for the set $X_i$, we continue extending the $(m,D)$-extendable subgraph used for $X_{i-1}$, repeating the exact same argument as before, only bearing in mind that we have already used up some vertices in $B$ for the previous steps of the embedding.

The only thing we need to verify is that the total number of vertices in the forest which we construct is always at most $s$. But this is straightforward to see, as the total number of vertices in the current $(D/2{-}1)$-ary trees which we use in every step is at most $s/2$ by construction. On the other hand, the total number of vertices in paths used for all $X_i$, by (\ref{total number in paths}), is at most $${d^*}\frac{100n\ln d} {\varepsilon^2d\ln C}\leq  \frac{100 n}{\varepsilon^2\ln C}< \frac{n}{4}\leq \frac{s}{2},$$ where we use that $C\geq e^{400/\varepsilon^2}$.

To summarize, we managed to connect all components in each dominating set $S_i$, $i\in [d^*]$, using internally vertex-disjoint paths whose internal vertices are in $B$. Thus we have constructed $d^*$ disjoint connected dominating sets in $G$, as required.

\section{Concluding remarks}

We have shown that $(n, d, \lambda)$-graphs, under modest assumptions on the spectral ratio $d/\lambda$, contain at least $(1 - o_d(1))d / \log d$ disjoint connected dominating sets, which is asymptotically optimal. This result demonstrates that robustly connected graphs can support a large number of disjoint connected dominating sets (CDSs). Moreover, our results imply that random $d$-regular graphs, and the binomial random graph $G(n, \frac{d}{n})$, for $d \gg \log n$, also typically contain $(1 - o_d(1))d / \log d$ disjoint CDSs, which matches the asymptotic lower bound.

As discussed in the introduction, \cite{censor2017tight} established that $k$-connected graphs contain $\Omega(k / \log^2 n)$ disjoint CDSs, while \cite{censor2014new} showed the existence of $k$-connected graphs with at most $O(k / \log n)$ disjoint CDSs. It would be interesting to determine whether a matching lower bound can be achieved. As a potentially more accessible problem, one could try to prove such a bound for larger values of $k = k(n)$. It is instructive to recall that that the so-called algebraic connectivity of a graph $G$, equal to $d-\lambda_2(G)$ in the case of $d$-regular graphs, is well-known to lower bound the standard vertex connectivity when $d<n-1$ (see, e.g., \cite{brouwer2011spectra}). Hence, in our context for $(n,d,\lambda)$-graphs $G$ with $\lambda\ll d$, the connectivity of $G$ is close to its degree $d$.

Recall the Independent Spanning Trees conjecture by Itai and Zehavi~\cite{zehavi1989three}, which asserts that every $k$-connected graph $G$ contains $k$ vertex-independent spanning trees. These are defined as $k$ spanning trees rooted at a node $r \in V(G)$, where the paths from $r$ to each vertex $v \in V(G)$ in different trees are internally vertex-disjoint. Our theorem implies that $(n, d, \lambda)$-graphs with mild conditions on the spectral ratio contain $(1 + o_d(1)) \frac{d}{\ln d}$ vertex-independent spanning trees, which is tight up to the $\ln d$ factor. This result extends to random $d$-regular graphs $G\sim G(n,d)$, as well as the binomial random graph $G(n, \frac{d}{n})$ for $d \gg \log n$. It would be interesting to prove the Itai-Zehavi conjecture for those random graphs, with the next step being to find $\Theta(d)$ vertex-independent spanning trees.




\vspace{0.66cm}

\noindent \textbf{Acknowledgement.} Part of this work was performed when the second author was visiting the Mathematical Institute and Merton College of the University of Oxford. He wishes to express his gratitude for the hospitality and great scientific atmosphere of these institutions.

\end{document}